\documentclass{loli}

\usepackage{amsfonts,amsmath,amsthm,amssymb,amscd}

\newtheorem{MainTheorem}{Theorem}
\newtheorem{Thm}{Theorem}[section]
\newtheorem{Lem}[Thm]{Lemma}
\newtheorem{Cor}[Thm]{Corollary}

\theoremstyle{definition}
\newtheorem{Def}[Thm]{Definition}

\theoremstyle{remark}
\newtheorem{Rem}[Thm]{Remark}

\numberwithin{equation}{section}

\def\Leb{{\mbox{Leb}}}
\def\dim{{\mbox{\footnotesize{dim}}}}

\begin{document}

\title[Weak Pseudo-Physical Measures and Pesin Formula for $C^1$-Anosov]{Weak Pseudo-Physical Measures and Pesin's Entropy  Formula
 for Anosov $C^1$-diffeomorphisms.}

\author{Eleonora Catsigeras}
\address{Instituto de Matem\'{a}tica y Estad\'{\i}stica \lq\lq Rafael Laguardia\rq\rq (IMERL),
Facultad de Ingenier\'{\i}a,
Universidad de la Rep\'{u}blica. Uruguay. }
\email{eleonora@fing.edu.uy}
\thanks{E.C. thanks the Department of Mathematics of the University of Alabama at Birmingham,   Prof. Alexander Blokh , and the scientific and organization committees  of \em The Dynamical Systems, Ergodic Theory and Probability Conference Dedicated to the Memory of Nikolai Chernov. \em E.C. and M.C. are partially supported by
Agencia Nacional de Investigaci\'{o}n e Innovaci\'{o}n (ANII), and Comisi\'{o}n Sectorial de Investigaci\'{o}n Cient\'{\i}fica (CSIC), Uruguay.
  E.C. is also supported in part by {\em \lq\lq For Women in Science\rq\rq} of L'Or\'{e}al-Unesco.}

\author{Marcelo Cerminara}
\email{cerminar@fing.edu.uy}

\author{Heber Enrich}
\email{enrich@fing.edu.uy}




\begin{abstract}
We consider   $C^1$ Anosov diffeomorphisms  on a compact Riemannian manifold. We define the weak pseudo-physical measures, which include the physical measures when these latter exist. We prove that ergodic weak pseudo-physical measures do  exist, and that the set of   invariant probability measures that satisfy   Pesin's Entropy Formula is the weak$^*$-closed convex hull of the ergodic   weak pseudo-physical measures. In brief, we give   in the $C^1$-scenario of uniform hyperbolicity,  a characterization of  Pesin's Entropy Formula in terms of physical-like properties.
\end{abstract}

\maketitle

%

\vspace{-.5cm}

{\footnotesize
\noindent 2010 {\em Math. Subj. Class.:} {Primary 37A35,  37D20. Secondary 37D35;  37A05.}

}

\section{Introduction}

 \label{sectionIntroduction}

 The   purpose of this paper is to give, in the $C^1$-scenario of uniform hyperbolicity, a characterization of those invariant measures that satisfy Pesin's Entropy Formula  in terms of their physical-like properties. Our main result   works, for $C^1$ Anosov diffeomorphisms, as Ledrappier-Young characterization  \cite{LedrappierYoung} of the measures $\mu$ that satisfy Pesin's Entropy Formula (which holds in the $C^2$ context but not in the general $C^1$ context), by substituting the property of absolute continuity of the unstable conditional decomposition of $\mu$, by the weak pseudo-physical property   of its ergodic decomposition.

 Pesin Theory \cite{Pesin1976, Pesin1977}  gives   relevant tools and results of the modern differentiable ergodic theory. It works for $C^2$  (or at least $C^1$ plus H\"{o}lder) dynamical systems.  For instance, for $C^2$ hyperbolic systems,  Pesin's Entropy Formula computes exactly the metric entropy of a diffeomorphism in terms of the  mean value of the sum of its positive Lyaypunov exponents. In the $C^2$ scenario, Pesin's Entropy Formula holds  if and only if the invariant measure has absolutely continuous conditional decomposition along the unstable manifolds. \cite{LedrappierYoung}.

 Through the properties of absolute continuity of invariant measures, and mainly through the absolute continuity of the holonomy along the invariant foliations,  Pesin Theory gives the   tools to construct   physically significant invariant measures for $C^1$-plus H\"{o}lder systems. Among these measures, the so called Sinai-Ruelle-Bowen (SRB) measures \cite{SinaiSRB,RuelleSRB,BowenSRB}, have particular relevance to describe the asymptotic statistics of Lebesgue-positive set of  orbits, not only for $C^1$ plus H\"{o}lder uniform and non-uniform hyperbolic systems, but also for $C^1$ plus H\"{o}lder partially hyperbolic systems \cite{PesinSinai,BDV}. Precisely, one of the most relevant properties of ergodic SRB measures for  $C^1$ plus H\"{o}lder  hyperbolic systems,   is that they are \em physical; \em  namely, their basins of statistical attraction have positive Lebesgue measure, even for Lebesgue non-preserving systems.

 In particular, for transitive Anosov $C^1$ plus H\"{o}lder  systems, the theorem of Pesin-Sinai  (see for instance \cite{PesinSinai}) states that there exists a unique physical     measure: it is the unique invariant probability measure that satisfies Pesin's Entropy Formula, and so, the only one with absolutely continuous conditional measures along the unstable foliation. Besides, its basin of statistical attraction covers Lebesgue all the orbits. In other words, for   a $C^1$ plus H\"{o}lder Anosov system, the definition of   physical measure,   SRB measure, and measure that satisfies Pesin's Entropy Formula, are equivalent.

 Nevertheless, in the $C^1$-scenario, the above results do not work,  because the theorems of   Pesin  Theory that ensure the  absolute continuity of unstable conditional measures, and of the holonomies along invariant foliations, fail. Even the existence of the unstable manifolds,  along which one could construct the conditional unstable measures, fails in the $C^1$ context  \cite{Pugh}. In the particular case of $C^1$-Anosov diffeomorphisms, invariant $C^0$ foliations with $C^1$ leaves do exist (see for instance \cite{HirshPughShub}), but the holonomies   along the invariant foliations are not necessarily absolutely continuous \cite{RobinsonYoung}.

 As a consequence, for   $C^1$   systems, if one defined SRB measures  by the existence of their absolutely continuous unstable conditional measures, one would  lack the hope to construct them. Nevertheless, one can still define   SRB or SRB-like measures, if one forgets for a while the properties of  absolute continuity, and focus the attention of the properties of statistical attraction.  In other words, one can try to look directly at their physical properties,   dodging  the lack  of conditional absolute continuity. For that reason, in the $C^1$-scenario, we    look for the Lebesgue abundance of points in their statistical basins  (or, more precisely, in the $\epsilon$-approach of their statistical basins). This search  was used in \cite{CE} to construct the SRB-like or pseudo-physical measures for $C^0$ dynamical systems on a compact Riemannian manifold. The notion of SRB-like or pseudo-physical measures, even in a non differentiable context, translate to the space of probability measures     the concept of statistical attraction   defined by Ilyashenko  \cite{GorodetskiIlyashenko,Ilyashenko} in the ambient manifold.

   In \cite{Qiu} it is proved   that $C^1$ \em generically, \em  transitive and uniformly hyperbolic systems  do have a unique measure satisfying Pesin's Entropy Formula. Besides this measure is physical and   its basin of statistical attraction covers Lebesgue almost all the orbits.
In the general  $C^1$-scenario with non uniform hyperbolicity,  Pesin's Entropy Formula was first proved considering   systems that preserve a  smooth measure (we call a measure $\mu$   smooth  if $\mu \ll \Leb$, where $\Leb$ is the Lebesgue measure). In fact, for $C^1$ generic diffeomorphisms that preserve a smooth measure $\mu$, \cite{Tahzibi} proved that $\mu$ satisfies Pesin's Entropy Formula. Later, in \cite{SunTian}, this formula was also proved for any  $C^1$ partially hyperbolic system  that preserves a smooth measure.

  If no smooth measure is preserved, in \cite{CCE} is proved  that the   pseudo-physical or SRB-like measures still exist and satisfy Pesin's Entropy Formula, provided that the system is   $C^1$ partially hyperbolic. Recently, and also for $C^1$ partially hyperbolic systems,   \cite{YangCao}   derived a proof of Shub's Entropy Conjecture \cite{Shub}  from their  method of construction of  measures that satisfy Pesin's Entropy Formula.

  In this paper we focus on $C^1$ Anosov systems to  search for a converse of the result in \cite{CCE}. Namely, our purpose is to characterize all the invariant measures that satisfy Pesin's Entropy Formula. First, we need  to generalize the concept of pseudo-physical or SRB-like measure. So, we define the \em weak pseudo physical \em   measures $\mu$, by taking into account only the $\epsilon$-approach  of its basin of statistical attraction up to time $n$, which we denote by $A_{\epsilon, \, n}(\mu)$, and the exponential rate of the variation of the Lebesgue measure of $A_{\epsilon, \, n}(\mu)$  when $n \rightarrow + \infty$  (Definition \ref{definitionWeakPseudoPhysicalMeasure}). In Theorem \ref{propositionWeaklyPseudoPhysicalMeasures} we study general properties of the weak pseudo physical measures, which do always exist.  We    prove that for any $C^1$  Anosov diffeomorphism, the   weak pseudo-physical measures satisfy Pesin's Entropy Formula (Part  \textsc{a}) of Theorem \ref{mainTheo}).
   Besides, we prove a converse result, to conclude that the set of invariant measures that satisfy Pesin's Entropy Formula is the closed convex hull of the weak pseudo-physical measures (Theorem \ref{mainTheo}).

   So, Theorem \ref{mainTheo} characterizes all the measures that satisfy Pesin's Entropy Formula in terms of the statistical properties that define the weak pseudo-physical notion. Nevertheless,  as far as we know, no example is still known of a $C^1 $ Anosov diffeomorphism for which weak pseudo-physical measures are not physical. In other words, there are not known examples of $C^1$-Anosov systems such that  an ergodic measure  satisfies Pesin's Entropy Formula and is non physical.

  The   proof of Theorem \ref{mainTheo} is based on the construction of local $C^1$ pseudo-unstable foliations, which approach  the local $C^0$ unstable foliation, and allow us to apply a Fubini decomposition of the Lebesgue measure of the $\epsilon$-basin $A_{\epsilon, \, n}(\mu)$,  for any ergodic measure $\mu$. The pseudo-unstable foliations are constructed via Hadamard graphs whose future iterates have bounded dispersion. This method was  introduced by Ma\~{n}\'{e} in \cite{Mane} to prove  Pesin's Entropy Formula in the $C^1$ plus H\"{o}lder context. Much later, it  was   applied also to $C^1$ systems  in \cite{SunTian} and \cite{CCE}.

  As said above, in Theorem \ref{mainTheo} of this paper, we prove that the weak pseudo-physical condition for the ergodic components of an invariant measure, is necessary and sufficient   to satisfy Pesin's Entropy Formula. The sufficient condition is just a corollary of the results in \cite{CCE}. On the contrary, the proof of the necessary condition  is new, although it is also strongly based on Ma\~{n}\'{e}'s method  to construct,  via Hadamard graphs, the $C^1$-pseudo unstable foliations.

  As a subproduct of the proof of Theorem \ref{mainTheo}, we also obtain an equality for any ergodic measure of a $C^1$-Anosov diffeomorphism, even for measures that do  not satisfy Pesin's Entropy Formula. This equality, which is stated in Theorem \ref{mainTheorem3}, considers the exponential rate $$a(\mu) := \lim_{\epsilon \rightarrow 0^+} \limsup_{n \rightarrow + \infty}\frac{ \log\mbox{Leb}(A_{\epsilon, n}(\mu))}{n},$$ according to which the Lebesgue measure of the  $\epsilon$-basin    $A_{\epsilon, n}(\mu)$ of each ergodic measure $\mu$ varies with time $n$. Theorem \ref{mainTheorem3} equals the exponential rate $a(\mu)$ with the difference $$h_{\mu} (f) - \int \sum_i \chi_i^+ \, d \mu,$$ where $h_{\mu}(f)$ is the metric entropy and $\{\chi_i^+\}$ are the positive Lyapunov exponents. So, in the particular case of ergodic measures $\mu$ satisfying Pesin's Entropy Formula, the exponential rate $a(\mu)$ is null, and conversely.

 \subsection{Definitions and Statement of the Results}

 Let $M$ be a compact, connected, Riemannian $C^1$-manifold without boundary and let $f: M \mapsto M$ be continuous.

 \begin{Def} \em  \em \label{definitionEmpiricProba} \textsc { (Empiric Probability)}
 For each $x \in M, \, n \in {\mathbb N}^+$,   the \em empiric probability \em $\sigma_n(x)$ along the finite piece of the future orbit of $x$  up to time $n$, is defined by
 $$\sigma_n(x) := \frac{1}{n}\sum_{j= 0}^{n-1} \delta_{f^j(x)},$$
 where $\delta_y$ denotes the Dirac delta probability measure supported on the point $y \in M$.
 \end{Def}

We denote by ${\mathcal M}$ the space of all the Borel probability measures on $M$, endowed with the weak$^*$ topology. We denote by ${\mathcal M}_f \subset {\mathcal M}$ the space of $f$-invariant Borel probability measures.
It is well known that ${\mathcal M}$ and ${\mathcal M}_f$ are nonempty, weak$^*$ compact, metrizable, sequentially compact and convex topological spaces. We fix and choose a metric $\mbox{dist}^*$ in ${\mathcal M}$ that induces the weak$^*$ topology.

\begin{Def} \em
\em \label{definitionBasinOfAttraction}  \textsc { (Basin  and pseudo basin  of attraction of a measure.)}
Let $\mu \in{\mathcal M}$ and $\epsilon >0$. We construct the following measurable sets in the manifold $M$:
 \begin{equation} \label{eqnL14} B(\mu) := \Big\{x \in M: \lim_{n \rightarrow + \infty} \sigma_n(x) = \mu   \Big\};\end{equation}

 \begin{equation} \label{eqnL15}A_{\epsilon}(\mu) := \Big\{x \in M: \liminf_{n \rightarrow + \infty} \mbox{dist}^*(\sigma_{n}(x), \mu) < \epsilon  \Big\};\end{equation}
 \begin{equation} \label{eqnL16}A_{\epsilon, \, n} (\mu) := \Big\{x \in M:   \mbox{dist}^* (\sigma_{n}(x), \mu) < \epsilon  \Big\}.\end{equation}

 We call $B(\mu)$ the \em  basin of attraction \em of $\mu$. We call $A_{\epsilon}(\mu)$ the \em $\epsilon$-pseudo basin of attraction  \em of $\mu$. We call $A_{\epsilon, \, n}(\mu)$ the \em $\epsilon$-pseudo basin of $\mu$ up to time $n$.\em

\end{Def}
In the sequel we denote by $\Leb$   the Lebesgue measure of $M$, renormalized to be a probability measure.
\begin{Def} \em
\em \label{definitionWeakPseudoPhysicalMeasure}  \textsc { (Physical, Pseudo-Physical and Weak Pseudo-Phys\-ical Measures)}

  Let $\mu \in {\mathcal M}$.
We call $\mu$  \em physical \em if $\Leb(B(\mu)) >0$.

  We call $\mu$   \em pseudo-physical \em  if $\Leb(A_{\epsilon}(\mu)) >0$ for all $\epsilon >0$.

  We call $\mu$   \em weak pseudo-physical \em if
\begin{equation} \label{eqnL12} \limsup_{n \rightarrow + \infty}\frac{1}{n} \log \Leb (A_{\epsilon, \, n}(\mu)) = 0 \ \ \ \forall \ \epsilon >0. \end{equation}
We denote   $${\mathcal P}_f := \big\{\mu \in {\mathcal M}\colon \mu \mbox{ is weak pseudo-physical} \big\}. $$
\end{Def}

On the one hand, it is standard to check that for continuous mappings $f: M \mapsto M$  the physical measures, if they exist, are $f$-invariant. Also, pseudo-physical are $f$-invariant (see \cite{CE}), page 153, and  as proved in Theorem 1.3 of \cite{CE},     the set of pseudo-physical measures is never empty, weak$^*$ compact and independent of the chosen metric $\mbox{dist}^*$ that induces the weak$^*$ topology of ${\mathcal M}$. Besides, it is immediate to check that physical measures, if they exist, are particular cases of the always existing pseudo-physical measures.

On the other hand,  in this paper we  will  generalize the previous results that hold for pseudo-physical measures, by proving the following properties also for weak pseudo-physical measures:

 \begin{MainTheorem}
 \label{propositionWeaklyPseudoPhysicalMeasures} Let $f: M \mapsto M$ be a continuous map. Then:

 \noindent \textsc {a) } Weak pseudo-physical measures are  $f$-invariant.

\noindent \textsc {b) } Physical measures and pseudo-physical measures are particular cases of weak pseudo-physical measures.

\noindent \textsc {c) } Weak pseudo-physical measures  do always exist.

\noindent \textsc {d) } The set ${\mathcal P}_f$ of weak pseudo-physical measures does not depend on the choice of the metric $\mbox{dist}^*$ that induces the weak$^*$ topology on ${\mathcal M}$.

\noindent \textsc {e) } ${\mathcal P}_f$ is weak$^*$-compact, hence sequentially compact.

\noindent \textsc {f) }   $ \displaystyle \lim _{n \rightarrow + \infty} \mbox{dist}^*(\sigma_n(x), {\mathcal P}_f)= 0 $  for Lebesgue almost all $x \in M$.

\noindent \textsc {g) } If the weak pseudo-physical measure $\mu$ is unique, then it is physical  and its basin of attraction $B(\mu)$ covers Lebesgue a.e. $x \in M$.

\end{MainTheorem}

\begin{Rem} \em
\em \label{RemarkPseudoPhysicalNoErgodica} Weak pseudo-physical measures \em are not necessaritly ergodic  \em (see   Example 5.4 of \cite{CE}).
\end{Rem}

Now, let $f \in \mbox{Diff}^1(M)$ be a $C^1$ diffeomorphism on $M$.

\begin{Def}
\label{DefinitionAnosov}  \textsc { (Anosov diffeomorphisms)}
The diffeomorphism $f$ is called \em Anosov \em if there exists a Riemannian metric of $M$ and asplitting $TM= E \oplus F$ which   is continuous and non trivial (i.e. $\mbox{dim}(E), \mbox{dim}(F) \neq 0$), and   a constant $\lambda <1$, such that
\begin{equation} \label{eqnL11} {\|Df_x|_{E(x)}\|} , \ \ \  { \|Df^{-1}_x|_{F(x)}\|}\leq \lambda \ \  \forall \ x\in M.\end{equation}
We call $E$ and $F $ the \em stable and unstable subbundles \em respectively. We call $\lambda$ the \em (uniform) hyperbolicity constant. \em
\end{Def}

 \begin{Rem} \em
 \label{RemarkAnosov} \em We observe that the condition of continuity of the unstable and stable subbundles is redundant in Definition \ref{DefinitionAnosov}. Besides, since the manifold is connected, from the continuity of $F$ and $E$ we deduce that they are uniformly transversal sub-bundles and  $\mbox{dim}(F)$ and $\mbox{dim}E$ are constants.

 From inequalities (\ref{eqnL11}), for any Anosov diffeomorphism $f$ and for any regular point $x \in M$, the minimum Lyapunov exponent along $F(x)$ is not smaller  than $\log \lambda^{-1} >0$, and  the maximum Lyapunov exponent along $E$ is not larger than $\log \lambda < 0$. Thus, for any regular point $x \in M$    all the Lyapunov exponents along $E(x)$ are strictly negative and bounded away from zero, and all the Lyapunov exponents along $F(x)$ are strictly positive and bounded away from zero. \end{Rem}

\begin{Def} \em
 \em \label{definitionPesinFormula}  \textsc { (Pesin's Entropy Formula)}
 Let $f \in \mbox{Diff}^1(M)$. Let $\mu \in {\mathcal M}_f$. We say that $\mu$ \em satisfies Pesin's Entropy Formula \em if
 $$h_{\mu}(f) = \int \sum_{i=1}^{   \mbox{\footnotesize dim}(M)} \chi^+_i(x)\, d \mu,$$
 where $h_{\mu}(f)$ is the metric entropy of $f$ with respect to $\mu$; for $\mu$-a.e. $x \in M$ the Lyapunov exponents of the orbit of $x$ are denoted by $$\chi_1(x) \geq \chi_2(x) \geq \ldots \geq       \chi_{\dim M}(x);$$ and   $\chi_i^+(x):= \max\{\chi_i(x), 0\}$.
\end{Def}

Recall that for  any $C^1$- Anosov diffeomorphism $f$,  the set of measures that satisfy Pesin's Entropy Formula is nonempty  (see for example Theorems 4.2.3 and 4.5.6 of \cite{Keller}).

\vspace{.3cm}

The main purpose of this paper is to prove the following result:

\begin{MainTheorem}
\label{mainTheo} For $C^1$ Anosov diffeomorphisms, the set of ergodic weak pseudo-physical measures is nonempty, and the set of invariant probability measures that satisfy Pesin's Entropy Formula is its closed convex hull.

\vspace{.5cm}

\em

 \noindent The following is an equivalent restatement of Theorem \ref{mainTheo}:

 \vspace{.2cm}

 \em

\noindent  \textsc {a) } All the weak pseudo-physical measures satisfy Pesin's Entropy Formula.

\noindent  \textsc {b) } Any  invariant probability measure  $\mu$ satisfies    Pesin's Entropy Formula if and only if its ergodic components $\mu_x$ are weak pseudo-physical $\mu$-a.e. $x \in M$. \em

\end{MainTheorem}

From Theorem \ref{mainTheo}, we obtain the following consequence:

\begin{Cor}
\label{corollaryC2-BonattiDiazViana}
If $f \in \mbox{\em Diff}^{1}(M)$ is Anosov, then
for Lebesgue-almost all $x \in M$  any convergent subsequence of the empirical probabilities $\sigma_n(x)$ converges to a measure that satisfies Pesin's Entropy Formula.
\end{Cor}
\begin{proof}
 From Assertion  \textsc{f}) of Theorem \ref{propositionWeaklyPseudoPhysicalMeasures}, for Lebesgue-almost all $x \in M$  any convergent subsequence of $\{\sigma_n(x)\}_{n \geq 1}$ converges to a weak pseudo-physical  measure $\mu$. Thus, applying part  \textsc{a}) of Theorem \ref{mainTheo} $\mu$ satisfies Pesin's Entropy Formula.
\end{proof}

The arguments to prove Theorem \ref{mainTheo} are based in the following more general result, which we will prove along the paper:

\begin{MainTheorem}
\label{mainTheorem3} If $f \in \mbox{\em Diff}^1(M)$ is   Anosov, if $F$ denotes its unstable sub-bundle, and if $\mu$ is an ergodic probability measure for $f$, then the $\epsilon$-pseudo basin  $A_{\epsilon, n}(\mu)$ of $\mu$ up to time $n$ satisfies the following equality: \em
\begin{equation}
\label{eqnL50} \lim_{\epsilon \rightarrow 0^+} \limsup_{n \rightarrow + \infty} \frac{\log \Leb (A_{\epsilon, \, n}(\mu))}{n} = h_{\mu}(f) - \int \log|\det Df |_{F } \, d \mu.
\end{equation}
\end{MainTheorem}

\subsection{Organization of the paper}
In Section \ref{sectionPropertiesWeaklyPseudoPhysicalMeasures} we prove Theorem \ref{propositionWeaklyPseudoPhysicalMeasures}, which states the general properties of weak  pseudo physical measures for any continuous map $f: M \mapsto M$.

In Section \ref{sectionSufficientCondition}, for Anosov diffeomorphisms,  we prove part  \textsc{a}) of Theorem \ref{mainTheo} and also the first part of  \textsc{b}). Precisely, we prove that the weak pseudo-physical property of the ergodic components is a sufficient condition to satisfy Pesin's Entropy Formula.

In Section \ref{sectionNecessaryCondition},  for Anosov diffeomorphisms, we prove the converse statement in  part \textsc{b}) of  Theorem \ref{mainTheo}. Namely, the weak pseudo-physical property of the ergodic components is also a  necessary condition to satisfy Pesin's Entropy Formula.

Through the proof of Theorem \ref{mainTheo}, we obtain some stronger intermediate results that hold for any ergodic measure. Finally, at the end of Section \ref{sectionNecessaryCondition}, we join those intermediate results to  prove Theorem \ref{mainTheorem3}.


\section{Properties of the weak pseudo-physical measures}
\label{sectionPropertiesWeaklyPseudoPhysicalMeasures}

The purpose of this section is to prove Theorem \ref{propositionWeaklyPseudoPhysicalMeasures}. Along this section, we   assume that $f$ is only a continuous map from a compact Riemannian manifold $M$ into itself.

Let us divide the proof  of Theorem \ref{propositionWeaklyPseudoPhysicalMeasures}   into its assertions  \textsc{a}) to \textsc{f}):

\vspace{.3cm}

\noindent \textsc { Theorem \ref{propositionWeaklyPseudoPhysicalMeasures} a) } \em Any weak pseudo-physical measure $\mu$ is $f$-invariant. \em
\begin{proof}
 From Equality (\ref{eqnL12}), for any fixed value of $\epsilon >0$ there exists $n_j \rightarrow + \infty$  such that $\Leb(A_{\epsilon, \, n_j}(\mu)) >0$. Thus, there exists $x_j \in M$ such that \begin{equation} \label{eqnL13} \mbox{dist}^*(\sigma_{n_j}(x_j), \mu) < \epsilon.\end{equation}
Since $\sigma_{n_j}(x_j) \in {\mathcal M}$ and ${\mathcal M}$ is sequentially compact, it is not restrictive to assume that
$\{\sigma_{n_j}(x_j)\}$ is weak$^*$ convergent. Denote by $\nu$ its limit. We assert that $\nu$ is $f$-invariant. In fact, consider  the operator $f^*:{\mathcal M} \mapsto {\mathcal M}$ defined by $f^*(\nu) (B) = \nu(f^{-1}(B))$ for any Borel measurable set $B \subset M$. Then $f^* (\delta_y) = \delta_{f(y)}$ for all $y \in M$; hence $f^* (\sigma_{n_j}(x_j)) = \sigma_{n_j}(f(x_j))$ for all $j \in \mathbb{N}$.

It is well known that $f^*$ is continuous. Thus, taking limit in the weak$^*$ topology, we obtain:
$$f^*(\nu)  = \lim_{j \rightarrow + \infty} f^* (\sigma_{n_j}(x_j)) = \lim_{j \rightarrow + \infty} \sigma_{n_j}(f(x_j)). $$
Since $$\sigma_{n_j} (x_j) = \frac{1}{n_j} \sum_{i= 1}^{n_j - 1}\delta_{f^i(x_j)} , \ \ \ \ \ \sigma_{n_j} (f(x_j)) = \frac{1}{n_j} \sum_{i= 1}^{n_j - 1}\delta_{f^{i+1}(x_j)},$$
we deduce that the total variation of the signed measure $\sigma_{n_j}(x_j) - f^*(\sigma_{n_j}(x_j))$ is
$$|\sigma_{n_j}(x_j) - f^*(\sigma_{n_j}(x_j))| \leq \frac{1}{n_j} \Big( \delta_{x_j} + \delta_{f^{n_j}(x_j)}\Big).$$
Thus   $\lim_{j \rightarrow + \infty} \sigma_{n_j}(x_j)= \lim_{j \rightarrow + \infty} f^*(\sigma_{n_j}(x_j)), $
hence
 $\nu = f^*(\nu)$, or equivalently $\nu$ is $f$-invariant.

From (\ref{eqnL13}) $\mbox{dist}^*(\nu, \mu) \leq \epsilon$. We have proved that for all $\epsilon >0$ there exists $\nu \in {\mathcal M}_f$ such that $\mbox{dist}^*(\mu, \nu) \leq \epsilon$. Since ${\mathcal M}_f$ is sequentially compact, we deduce that $\mu \in {\mathcal M}_f$, as wanted. \end{proof}

\vspace{.3cm}

\noindent \textsc { Theorem \ref{propositionWeaklyPseudoPhysicalMeasures} b) } \em Any physical or pseudo-physical measure is weak pseudo-physical. \em
\begin{proof}
Trivially any physical measure is pseudo-physical.  So, it is only left to prove  that any pseudo-physical measure $\mu$ is weak pseudo-physical. Consider $x \in A_{\epsilon}(\mu)$. From equality (\ref{eqnL15}), there exists $n_j \rightarrow + \infty$ such that $\mbox{dist}^*(\sigma_{n_j}(x), \mu) < \epsilon$. Therefore, from  (\ref{eqnL16})  $x \in \bigcap_{N  \geq 1} \bigcup_{n \geq N} A_{\epsilon, \, n}(\mu). $ Since the latter assertions holds for all $x \in A_{\epsilon}(\mu)$, we have proved that
$$A_{\epsilon}(\mu) \subset \bigcap_{N  \geq 1} \bigcup_{n \geq N} A_{\epsilon, \, n}(\mu).$$
As $\mu$ is pseudo-physical, we deduce:
\begin{equation} \label{eqnL17} \Leb\Big(\bigcap_{N  \geq 1} \bigcup_{n \geq N} A_{\epsilon, \, n}(\mu) \Big) \geq \Leb \Big(A_{\epsilon}(\mu) \Big) >0 \ \ \forall \epsilon >0.\end{equation}

Now, assume by contradiction, that $\mu$ is not weak pseudo-physical. Taking into account that $\Leb(A_{\epsilon, \, n}(\mu)) \leq 1$, from the contrary of equality (\ref{eqnL12}), we deduce that there exist  $\epsilon >0$ and $a >0$ such that
$$\limsup_{n \rightarrow + \infty}\frac{1}{n} \log \Leb (A_{\epsilon, \, n}(\mu)) = -2a < 0.$$
Therefore, there exists $N \geq 0$ such that
 $ \displaystyle \Leb(A_{\epsilon, \, n}(\mu) \leq e^{-an} $ for all $  n \geq N,$
from where we deduce that
 $\displaystyle \sum_{n= 1}^{+ \infty} \Leb(A_{\epsilon, \, n}(\mu) < + \infty.$
Finally, applying Borell-Cantelli Lemma, we conclude that
 $\displaystyle \Leb\Big(\bigcap_{N  \geq 1} \bigcup_{n \geq N} A_{\epsilon, \, n}(\mu) \Big) =0,$
contradicting inequality (\ref{eqnL17}).
\end{proof}

\vspace{.3cm}

\noindent \textsc { Theorem \ref{propositionWeaklyPseudoPhysicalMeasures} c) } \em Weak pseudo-physical measures do exist. \em
\begin{proof}
In Theorem 1.3 of \cite{CE}, it is proved   for any continuous map  $f$ on a compact manifold, that the pseudo-physical measures  (which in that paper are also called SRB-like or observable)  do   exist. Since any pseudo-physical measure is weak pseudo-physical,  these latter measures   always exist.\end{proof}

\vspace{.3cm}

\noindent \textsc { Theorem \ref{propositionWeaklyPseudoPhysicalMeasures} d) } \em The set ${\mathcal P}_f$ of weak pseudo-physical measures does not depend on the choice of the metric  in ${\mathcal M}$ that induces the weak$^*$ topology. \em
\begin{proof}
Take two metrics $\mbox{dist}^*_1$ and $\mbox{dist}^*_2$, both inducing the weak$^*$ topology on ${\mathcal M}$. We assume that $\mu $ is weak pseudo-physical according to $\mbox{dist}^*_1$, and let us prove that it is also weak pseudo-physical according to  $\mbox{dist}^*_2$.

Since both metric induce the same topology,
for any $\epsilon >0$ there exists $\delta >0$ such that
\begin{equation} \label{eqnL19} \rho \in {\mathcal M}, \ \ \mbox{dist}^*_1(\rho, \mu) < \delta \ \ \Rightarrow \ \ \mbox{dist}^*_2(\rho, \mu) < \epsilon.\end{equation}
In the notation of equality (\ref{eqnL16}), add a subindex 1 or 2  to denote the sets $A_{\cdot, \, n, \ 1}(\mu)$ and $A_{\cdot, \, n, \ 2}(\mu)$, according to which metric ($\mbox{dist}^*_1$  and $\mbox{dist}^*_2$, respectively) is used to define them.  So, from assertion (\ref{eqnL19}) we have:
 $A_{\delta, \, n, \ 1}(\mu) \subset A_{\epsilon, \, n, \ 2},$
from where
 $$\Leb \Big(A_{\delta, \, n, \ 1}(\mu)  \Big) \leq \Leb \Big(   A_{\epsilon, \, n, \ 2}(\mu)\Big).$$

Since we are assuming that $\mu$ is weak pseudo-physical according to $\mbox{dist}_1^*$,   from equality (\ref{eqnL12}) we know that
$$\limsup_{n \rightarrow + \infty} \frac{\log \Leb(A_{\delta, \, n, \ 1}(\mu))}{n} = 0.$$
Then, $$\limsup_{n \rightarrow + \infty} \frac{\log \Leb(A_{\epsilon, \, n, \ 2}(\mu))}{n} \geq 0. $$ As $\epsilon >0$ was arbitrarily chosen, the latter inequality holds for all $\epsilon >0$.
But the   limit in the latter inequality is non positive because $\Leb$ is a probability measure. We conclude that
$$\limsup_{n \rightarrow + \infty} \frac{\log \Leb(A_{\epsilon, \, n,\ 2}(\mu))}{n} = 0 \ \ \ \ \forall \ \epsilon >0, $$
ending the proof that $\mu$ is also weak pseudo-physical with respect to the metric $\mbox{dist}^*_2$.
\end{proof}

\vspace{.3cm}

\noindent \textsc { Theorem \ref{propositionWeaklyPseudoPhysicalMeasures} e) } \em The set ${\mathcal P}_f $ of weak pseudo-physical measures is weak$^*$-compact.
 \em
\begin{proof}
Since ${\mathcal P}_ f \subset {\mathcal M}$ and ${\mathcal M}$ is weak$^*$-compact, it is enough to prove that ${\mathcal P}_f$ is weak$^*$-closed. Assume $\mu_j \in {\mathcal P}_f$ and $\mu \in {\mathcal M}$ such that
$$\lim_{j \rightarrow +\infty} \mbox{dist}^* (\mu_j, \mu) = 0 $$
We will prove that $\mu \in {\mathcal P}_f$. For any given $\epsilon >0$, choose and fix $j$ such that $\mbox{dist}^*(\mu_j, \mu) < \epsilon/2$.  Thus, from equality (\ref{eqnL16}) and the triangle property, we obtain:
 $ A_{\epsilon/2, \, n} (\mu_j) \subset A_{\epsilon, \, n}(\mu),$
from where \begin{equation}  \label{eqnL18} \Leb \Big(A_{\epsilon/2, \, n} (\mu_j) \Big) \leq \Leb \Big(A_{\epsilon, \, n} (\mu) \Big).\end{equation}
Since $\mu_j \in {\mathcal P}_f$, we can apply equality (\ref{eqnL12}) to $\Leb \Big(A_{\epsilon/2, \, n} (\mu_j)\Big)$, which joint with inequality (\ref{eqnL18}) implies:
$$\limsup_{n \rightarrow + \infty} \frac{1}{n}\log \Leb \Big(A_{\epsilon, \, n} (\mu) \Big) \geq 0. $$
Finally, since $\Leb$ is a probability measure, we deduce that the above limsup equals 0, concluding that $\mu \in {\mathcal P}_f$ as wanted.
\end{proof}

\vspace{.3cm}

\noindent \textsc { Theorem \ref{propositionWeaklyPseudoPhysicalMeasures} f) } \em $ \displaystyle \lim _{n \rightarrow + \infty} \mbox{dist}^*(\sigma_n(x), {\mathcal P}_f)= 0 $  for Lebesgue almost all $x \in M$. \em
\begin{proof}
  Theorem 1.5 of \cite{CE}, states that the distance between $\sigma_n(x)$ and the set of pseudo-physical measures converges to zero with $n \rightarrow + \infty$ for Lebesgue almost all $x \in M$. Since the pseudo-physical measures are contained in ${\mathcal P}_f$, we trivially deduce the wanted equality. \end{proof}

\noindent \textsc { Theorem \ref{propositionWeaklyPseudoPhysicalMeasures} g) } \em  If the weak pseudo-physical measure $\mu$ is unique, then it is physical and its basin of attraction $B(\mu)$ covers Lebesgue a.e. $x \in M$. \em
\begin{proof}
  It is an immediate consequence of part \textsc{f)}. \end{proof}

\section{Sufficient condition for Pesin's Entropy Formula} \label{sectionSufficientCondition}

In the sequel we assume that the map $f\in \mbox{Diff}^1(M)$ is Anosov.  The purpose of this section is to deduce, as an immediate consequence from previous known results,  part  \textsc{a})  of Theorem \ref{mainTheo}, and  the sufficient condition to satisfy Pesin's Entropy Formula   in part  \textsc{b})  of Theorem \ref{mainTheo}. Namely, we will deduce  that if all the ergodic components of an $f$-invariant measure $\mu$  are weak pseudo-physical, then $\mu$ satisfies Pesin's Entropy Formula.

Recall Definition \ref{definitionBasinOfAttraction}, which defines the $\epsilon$-pseudo basin $A_{\epsilon, \, n}(\mu)$ up to time $n$  of a probability measure $\mu  $. We will   apply  the following  result:

\begin{Thm}
 \textsc {\cite{CCE}} \label{TheoremCCE1} Let $M$ be a compact Riemannian manifold of finite dimension.
Let $f \in \mbox{Diff}^1(M)$ be Anosov with  hyperbolic splitting $TM= E \oplus F$, where $E$ and $F$ are the stable and unstable sub-bundles respectively. Then, the following inequality holds for any $f$-invariant $\mu \in {\mathcal M}$:
\begin{equation}
\label{eqnL20} \lim_{\epsilon \rightarrow 0^+} \limsup_{n \rightarrow + \infty} \frac{\log \Leb (A_{\epsilon, \, n}(\mu))}{n} \leq h_{\mu}(f) - \int \log|\det Df |_{F } \, d \mu.
\end{equation}
\end{Thm}
{\textsc{Proof}. }
See Proposition 2.1 in \cite{CCE}.

\begin{Rem} \em  \label{RemarkRuelle}
 \em For the non negative Lyapunov exponents, we adopt the notation   $\chi^+_i(x)$ as in Definition \ref{definitionPesinFormula}. For any $f \in \mbox{Diff}^1(M)$, Margulis and Ruelle  inequality \cite{Margulis, Ruelle} states:
  \begin{equation} \label{eqnRuelleInequality}h_{\mu}(f) \leq \int \sum_{i=1}^{   \mbox{\footnotesize dim}(M)} \chi^+_i(x)\, d \mu.\end{equation}
  Thus, Pesin's Entropy Formula holds for an invariant  measure $\mu$, if and only if the following inequality holds:
  $$h_{\mu}(f) \geq \int \sum_{i=1}^{   \mbox{\footnotesize dim}(M)} \chi^+_i(x)\, d \mu.$$
  Besides, from Definition \ref{DefinitionAnosov}, and from the formula of the integral of the volume form along the unstable sub-bundle $F$, we obtain the following equality for Anosov diffeomorphisms:
   $$ \int \sum_{i=1}^{   \mbox{\footnotesize dim}(M)} \chi^+_i(x)\, d \mu =  \int \log|\det Df|_F| \, d \mu.$$
   Joining the above assertions, we conclude:

   \em Let $f \in \mbox{\em Diff}^1(M)$  be Anosov, and $F$ be its unstable sub-bundle. Then, any   $f$-invariant probability measure $\mu$  satisfies Pesin's Entropy Formula if and only if
   \begin{equation}
   \label{eqnL21}
   h_{\mu}(f) \geq \int \log |\det Df|_F| \, d \mu.\end{equation}
\end{Rem}

\vspace{.3cm}

We are ready to deduce part  \textsc{a})  of Theorem \ref{mainTheo}, which is indeed a corollary of Theorem \ref{TheoremCCE1}:

\vspace{.3cm}

\noindent \textsc { Part a) of Theorem \ref{mainTheo}: } \em If $f \in \mbox{\em Diff}^1(M)$ is Anosov and if $\mu$ is a weak pseudo-physical   $f$-invariant measure, then $\mu$ satisfies Pesin's Entropy Formula. Therefore, the set of invariant probability measures that satisfy Pesin's Entropy Formula  is nonempty. \em

\begin{proof}

 By contradiction, assume that $\mu$ does not satisfy Pesin's Entropy Formula. According to Remark \ref{RemarkRuelle}, inequality (\ref{eqnL21}) does not hold:
 $$h_{\mu}(f) - \int \log |\det Df|_F| \, d \mu <0.$$
 Therefore, applying inequality (\ref{eqnL20}) of Theorem \ref{TheoremCCE1}, we conclude that there exists $\epsilon >0$ such that
 $$\limsup_{n \rightarrow + \infty} \frac{\log \Leb (A_{\epsilon, \, n}(\mu))}{n}  <   0. $$
  So, equality (\ref{eqnL12}) does not hold; hence $\mu$ is not weak pseudo-physical, contradicting the hypothesis. We have proved that all the weak pseudo-physical measures for $f$ satisfy Pesin's Entropy Formula. From part  \textsc{c}) of Theorem \ref{propositionWeaklyPseudoPhysicalMeasures},  weak pseudo-physical measures do exist. So, the set of measures that satisfy   Pesin's Entropy Formula is nonempty.
 \end{proof}

 \vspace{.3cm}

 We now recall the following well known result (see for instance Theorems 4.3.7 and 4.5.6 of \cite{Keller}):
 \begin{Thm} \label{TheoremConvexHull}
Let $f \in \mbox{\em Diff}^1(M)$ be Anosov. An $f$-invariant measure $\mu$ satisfies Pesin's Entropy Formula if and only if its ergodic components $\mu_x$ satisfy it for $\mu$-a.e. $x \in M$.
\end{Thm}

\begin{proof} On the one hand, we recall that any Anosov $C^1$ diffeomorphism $f$ is expansive, and for any expansive homeomorphism $f$ on $M $  the metric entropy $h_{\mu}(f)$ depends upper semi-continuously on the $f$-invariant measure $\mu$ (see for instance Theorem 4.5.6 of \cite{Keller}).
 So, we can apply the theorem of the Affinity of the Entropy Function (see Theorem 4.3.7 of \cite{Keller}), which states that
 \begin{equation} \label{eqnL22} h_{\mu}(f) = \int h_{\mu_x}(f) \, d\mu(x),\end{equation}
 where the measures $ \mu_x $ for $\mu-\mbox{a.e. } x \in M $  are the ergodic components of $\mu$.

  On the other hand, the ergodic decomposition theorem states that
  \begin{equation} \label{eqnL23}\int  \log |\det Df|_F  |\, d \mu  = \int d \mu(x) \Big( \int|\det Df_y|_{F(y)}  | \, d \mu_x(y) \Big). \end{equation}

  Joining Equalities (\ref{eqnL22}) and (\ref{eqnL23}), and taking into account Margulis and Ruelle inequality (\ref{eqnRuelleInequality}), we deduce that
  $$h_{\mu}(f)  - \int  \log |\det Df|_F  |\, d \mu = 0 $$ if and only if
  $$h_{\mu_x}(f)  - \int  \log |\det Df(y)|_{F(y)}  |\, d \mu_x(y) = 0 \mbox{ for } \mu-\mbox{a.e. } x \in M,$$ ending the proof of Theorem \ref{TheoremConvexHull}.
\end{proof}

 As a consequence we obtain:

 \vspace{.3cm}

 \noindent \textsc { Part b) of Theorem \ref{mainTheo}, sufficient condition: } \em If $f \in \mbox{\em Diff}^1(M)$ is Anosov and if $\mu$ is an invariant measure whose ergodic components $\mu_x$ are weak pseudo-physical for $\mu$-a.e. $x \in M$, then $\mu$ satisfies Pesin's Entropy Formula.   \em

\begin{proof}
From the hypothesis, and applying part  \textsc{a}) of Theorem \ref{mainTheo}, we deduce that the ergodic components $\mu_x$ of $\mu$ satisfy Pesin's Entropy Formula for $\mu$-a.e. $x \in M$. So, from Theorem \ref{TheoremConvexHull}, the measure $\mu$  also satisfies this formula.
\end{proof}

\section{Necessary condition for Pesin's Entropy Formula} \label{sectionNecessaryCondition}

In this section we will prove the necessary condition   to satisfy Pesin's Entropy Formula, as stated in part  \textsc{b}) of Theorem \ref{mainTheo}. Precisely, we will prove  that if $f \in \mbox{Diff}^1(M)$ is Anosov, and if the $f$-invariant measure $\mu$ satisfies Pesin's Entropy Formula, then the ergodic components of   $\mu$  are weak pseudo-physical. We will also prove the equality of Theorem \ref{mainTheorem3} for any ergodic measure $\mu$.

\subsection{Previous known properties for Anosov diffeomorphisms}.

  \noindent \textsc { Expansivity.}
  Recall that any Anosov diffeomorphism $f$ is expansive (see for instance Lemma 3.4 in \cite{Bowen}). Namely, there exists a constant $\alpha >0$, which is called the expansivity constant, such that $$\mbox{dist}(f^n(x), f^n(y)) < \alpha \ \ \ \forall \, n \in \mathbb{Z} \ \ \ \Rightarrow \ \ \ x= y.$$

Given two partitions ${\mathcal Q}$ and ${\mathcal R}$,  the partition ${\mathcal Q} \vee {\mathcal R}$ is defined by
$${\mathcal Q} \vee {\mathcal R} = \Big\{Q \cap R \colon \ \ Q \in {\mathcal Q}, \ R \in {\mathcal R} \Big\}.$$

\noindent \textsc { Metric entropy for expansive systems.}
 Recall the following result, which follows from Kolmorgorov-Sinai Theorem  in the case of expansive homeomorphisms   (see for instance,  Proposition 2.5 of \cite{Bowen}, or also  Theorem 3.2.18 and Lemma 4.5.4 of \cite{Keller}):

 \em If ${\mathcal R}$ is a finite partition whose pieces are Borel measurable sets and have diameter smaller than the expansivity constant $\alpha $, then   $\bigcup_{n= 0}^{+ \infty}\{\bigvee_{k=-n }^{k= + n} f^{-j}{\mathcal R}\} $ generates the Borel $\sigma$-algebra, and for any $f$-invariant measure $\mu$, the metric entropy $h_{\mu}(f)$ can be computed by: \em
\begin{equation} \label{eqnMetricEntropy} h_{\mu}(f) = \limsup_{n \rightarrow + \infty} \frac{H({\mathcal R}_n, \mu)}{n}, \ \
\mbox{ where }  \   {\mathcal R}_n := \bigvee_{j= 0}^{n-1} f^{-j}({\mathcal R}), \ \ \mbox{ and } \end{equation} \begin{equation} \label{eqn39} H({\mathcal R}_n, \mu) := -\sum_{Y \in {\mathcal R}_n} \mu(Y) \log (\mu(Y)) \leq   \log \#\{Y \in {\mathcal R}_n \colon \mu(Y) >0\}. \end{equation}
  Note: In (\ref{eqn39}) at right,  $\#{\mathcal P}$ denotes the number of elements of the  finite set ${\mathcal P}$.

\vspace{.3cm}

\noindent  \textsc { Rectangles.}
Recall the definition of   rectangle $R$ in the manifold $M$ for the Anosov diffeomorphism $f$. (See \cite{Bowen}, page 78.) In particular a rectangle $R$ is proper if $\overline{R} =  R = \overline{\mbox{int}(R)}$.  For any proper rectangle $R$ and any $x \in R$   denote   $$W^s_R:= \mbox{connected component}(W^s(x) \cap R) \ni x, $$ $$ W^u_R:= \mbox{connected component}(W^u(x) \cap R) \ni x,$$ where $W^s(x), \ W^u(x)$ are the stable and unstable submanifolds of the point $x$.


\vspace{.3cm}

The    properties below follow from the definition of rectangle $R$:

\noindent  \textsc {(i) } For any pair of points $x, y \in R$ there exists a unique point, which we denote by $[x,y]$, defined by $$[x,y]:   \in W^s_R(x)  \pitchfork W^u_R(x).$$

\noindent  \textsc {(ii) } There exists a constant $K_R >0$ such that, if $L\subset R$ is a   local embedded $C^1$-submanifold,  with dimension equal to the unstable dimension, and such that $L$ intersects transversally  the local stable manifolds $W^s_{R}(x)$ for all $x \in R$, then   \begin{equation} \label{eqnL25} \mbox{Leb}^L(L) \geq K_R^{-1}, \end{equation} where $\mbox{Leb}^L$ denotes the Lebesgue measure along $L$.

 We recall the definition of Markov partition ${\mathcal R} = \{R_i\}_{1 \leq i \leq k}$ into rectangles $R_i$  (see \cite{Bowen}, pages 78--79) and the following well known result:
\begin{Thm}
\label{TheoremExistenceMarkovPartition}
 \textsc { (Existence of Markov Partitions).}

 Let $f \in \mbox{\em Diff}^1(M)$ be Anosov. Then, for all $\delta  >0$ there exists a Markov partition whose rectangles have diameter  smaller than $\delta$.
\end{Thm}
{\textsc{Proof}. }   See  Theorem 3.12 of \cite{Bowen}.

 \begin{Def} \em
 \label{definitionDynamicalRectangle} \textsc { (Dynamical rectangle)} \em

 Let ${\mathcal R} $ be a Markov partition of the manifold $M$, let $x \in M$ and denote by $R(x)$ the rectangle of ${\mathcal R}$ that contains $x$. Let $n  $ be a positive natural number.  The \em dynamical rectangle \em $R_n(x)$ that contains $x$ is defined by
 $$R_n(x) := \bigcap_{j= 0}^{n-1} f^{-j}(R(f^j(x))).$$

 The following property follows from the definition of Markov Partition (see Condition (b) in \cite{Bowen}, page 79):
 \begin{equation} \label{eqnDynamicalRectangle}
 \forall \  n \geq 0, \ \ \mbox{ if } y \in R_n(x)  \ \  \mbox{ then } \ \  W^s_{R}(y) \subset R_n(x).\end{equation}
\end{Def}

\subsection{Technical Lemmas}

\begin{Lem}
\label{lemmaEntropia} Let $f$ be an Anosov diffeomorphism on a compact manifold. Then,  for any finite Markov partition ${\mathcal R}$, there exists a constant $K_0 >0$ satisfying the following inequality for any $f$-invariant probability measure $\mu$, for any $0 < \epsilon < 1/4$, for   any Borel measurable set $A \subset M$ such that $\mu(A) > 1 - \epsilon$, and for any natural number $n \geq 1$:
\begin{equation} \label{eqnL40} \log \#\{Y \in {\mathcal R}_n  \colon Y \cap A \neq \emptyset\} \geq H({\mathcal R}_n , \mu) - n \cdot K_0 \cdot \epsilon  +  \epsilon\log \epsilon + (1- \epsilon) \log (1- \epsilon).\end{equation}

\end{Lem}


\begin{proof}   Denote    $A_n := \bigcup \Big \{ Y \in {\mathcal R}_n \colon A \cap Y  \neq \emptyset \Big\}.$
Since $A \subset A_n$ we have $\mu(A_n) > 1 - \epsilon$. If $\mu(A_n) = 1$ then inequality (\ref{eqnL40}) holds  trivially as a consequence of   (\ref{eqn39}). So, let us  prove Lemma \ref{lemmaEntropia} in the case $$1 - \epsilon <\mu(A_n) < 1; \mbox{ hence } 0 < \mu(M \setminus A_n) < \epsilon.$$
By definition: $$H({\mathcal R}_n, \mu) := - \sum_{Y \in {\mathcal R}_n} \mu(Y) \log \mu(Y)  $$ $$= -\sum_{Y \subset A_n}  \mu(Y) \log \mu(Y) - \sum_{Y \subset M \setminus A_n} \mu(Y) \log \mu(Y)$$ $$ = -\mu(A_n) \sum_{Y \subset A_n} \frac{\mu(Y)}{\mu(A_n)} \log \Big(\frac{\mu(Y)}{\mu(A_n)}     \Big) - \sum_{Y \subset A_n} \mu(Y) \log \mu(A_n)   $$
$$  -\mu(M \setminus A_n) \sum_{Y \subset M \setminus A_n} \frac{\mu(Y)}{\mu(M \setminus A_n)} \log \Big(\frac{\mu(Y)}{\mu(M \setminus A_n)}     \Big) - \sum_{Y \subset M \setminus A_n} \mu(Y) \log \mu(M \setminus A_n).$$

Construct the probability measures $\mu_1$ and $\mu_2$ defined by the following equalities for all Borelian set $B \subset M$: $$\mu_1(B) := \mu (B \cap A_n)/ \mu(B), \ \ \ \mu_2(B) := \mu(B \cap(M \setminus A_n))/\mu(M \setminus A_n).$$
We obtain
$$H({\mathcal R}_n, \mu)  =   -\mu(A_n) \sum_{Y \subset A_n}  {\mu_1(Y)} \log \Big( {\mu_1(Y)}     \Big) - \mu( A_n)   \log \mu(A_n)     $$ $$ -\mu(M \setminus A_n) \sum_{Y \subset M \setminus A_n}  {\mu_2(Y)}  \log \Big( {\mu_2(Y)}     \Big) -   \mu(M \setminus A_n) \log \mu(M \setminus A_n).$$
Applying inequality (\ref{eqn39}):
$$H({\mathcal R}_n, \mu) \leq     \log \#\big \{Y \in {\mathcal R}_n \colon  {Y \subset A_n}\big \} - \mu( A_n)   \log \mu(A_n)     $$ $$ + \mu(M \setminus A_n) \cdot \log \#  {\mathcal R}_n  -   \mu(M \setminus A_n) \log \mu(M \setminus A_n).$$
Taking into account that $0 <\mu(M \setminus A_n) < \epsilon < 1/4$ and that $-u\log u$ is strictly increasing for $0 < u < 1/4$ and strictly decreasing for $u > 3/4$, we obtain:
$$H({\mathcal R}_n, \mu) \leq  $$ $$  \log \#\big \{Y \in {\mathcal R}_n \colon  {Y \subset A_n}\big \} +  \mu(M \setminus A_n) \cdot \log \#  {\mathcal R}_n       -  \epsilon \log \epsilon -  (1- \epsilon) \log (1 - \epsilon) \leq $$ $$  \log \#\big \{Y \in {\mathcal R}_n \colon  {Y \subset A_n}\big \} +  \epsilon \cdot n \cdot  K_0 \cdot      -  \epsilon \log \epsilon -  (1- \epsilon) \log (1 - \epsilon),$$
where $$K_0 := \sup_{n \geq 1} \frac {\log \#  {\mathcal R}_n}{n} >0 $$
Therefore, to end the proof of Lemma \ref{lemmaEntropia} it is enough to show that $K_0 < + \infty$. In fact, for a Markov partition ${\mathcal R}$, any rectangle $Y \in {\mathcal R}_n$ is obtained  as a {connected component} of the intersection $ f^{-n}(R_i) \cap R_j $  for some pair of rectangles  $R_i, R_j \in {\mathcal R}$. Fixing $R_i \in {\mathcal R}$, the maximum number of connected components of the intersections of $f^{-n}(R_i)$ with the   rectangle  $R_j$ of the partition, is upper bounded the following quotient
$$   \frac{\max \Big \{\Leb(W^s_{f^{-n}(R_i)}(y)) \colon y \in {f^{-n}( {R_i})}\Big\}}{\min  \Big\{\Leb(W^s_{R_j}(x)) \colon x \in  {R_j} \Big\}}  \leq  $$ $$ d^n \frac{\max  \Big\{\Leb(W^s_{R_i}(x)) \colon x \in {R_i}\Big\}}{\min  \Big\{\Leb(W^s_{R_j}(x)) \colon x \in  {R_j}\Big\}} =: d^n \cdot q_{i,j},$$
where $$d := \max\Big\{ \big|\mbox{det} Df^{-1}_x|_{E(x)}\big| \colon x \in M \Big\}.$$     Thus, denoting   $k:= \#\mathcal R$, we have
$$ \#{\mathcal R}_n \leq k^2 \cdot d^n \cdot \max \Big \{  q_{i,j}: \ 1 \leq i,j \leq k \Big \}. $$
 We conclude that
 $$\limsup_{n \rightarrow + \infty} \frac {\log \#{\mathcal R}_n}{n} \leq   d   < + \infty,$$
which implies $ K_0:=  \sup _{n \geq 1} \displaystyle \frac {\log \#{\mathcal R}_n}{n}  < + \infty,$
ending the proof of Lemma \ref{lemmaEntropia}.
\end{proof}

 In the following Lemma    we will construct a local $C^1$-foliation, whose leaves are \em pseudo-unstable manifolds \em    $\epsilon$- approaching (in the $C^1$-topology)  the true local unstable manifolds of any rectangle of a  given Markov partition.

\begin{Lem} \label{lemmaFoliation} Let $f \in \mbox{\em Diff}^1(M)$ be Anosov. Denote  the stable and unstable subbundles by $E$ and $ F$, respectively. Denote  the expansivity constant by $\alpha >0$. Then,
for all $\epsilon >0$ there exist $0 <\delta_0  < \alpha$ and  $  K >0$ such that, for  any finite Markov partition ${\mathcal R} = \{R_i\}_{1 \leq i \leq k}$ into rectangles with diameter smaller than $\delta_0$, there exists a finite family    $\{{\mathcal L}_i\}_{1 \leq i \leq k}$ of local foliations ${\mathcal L}_i$, each one defined in an open neighborhood of each rectangle $R_i$,   satisfying the following properties  for all $1 \leq i \leq k$, for all $x \in R_i$ and for all $n \geq 0$:

\noindent \textsc {a) } ${\mathcal L}_i$ is $C^1$-trivializable and its leaves are $\mbox{dim}(F)$- dimensional.

\noindent \textsc {b) } \em $\mbox{dist}\big(F_{f^n(x)}, T_{f^n(x)}f^n({\mathcal L}_i(x)) \big) < \epsilon$, \em  where ${\mathcal L}_i(x)$ denotes the leaf of the foliation ${\mathcal L}_i$ that contains $x$.

\noindent \textsc {c) }
 $ \displaystyle K^{-1} e^{-n \epsilon} \leq \frac{\Big|\det dDf^n_x|_{\displaystyle T_x({\mathcal L_i}(x))}\Big| }{\Big|\det Df^n_x|_{\displaystyle F(x)}\Big|} \leq K \, e^{n \epsilon}. $

\noindent \textsc {d) }     There exist a point $x_i \in R_i$ and an open subset $A_i^s \subset W^s_{R_i}(x_i)$, in the topology of the  stable submanifold $W^s(x_i)$, such that \em
$$\Leb^{ W^s(x_i)} \big(A_i^s  \big ) \geq K^{-1}, $$ \em where \em $\Leb^{ W^s(x_i)}$ \em denotes the Lebesgue measure along the  submanifold $W^s(x_i)$; and besides, if $y \in A_i^s$, then \em
$$Leb^{ f^{n}({\mathcal L_i}(y))} \Big(f^n\big({\mathcal L_i}(y)  \cap    R_n(x)\big) \Big) \geq K^{-1},$$ \em
where \em $\Leb^{ f^{n}({\mathcal L_i}(y))}$ \em denotes the Lebesgue measure along the   submanifold ${\mathcal L_i}(y)$.
\end{Lem}
\begin{proof}
Proposition 3.6 of \cite{CCE} states the existence of $\delta_0>0$ and   the  local $C^1$-foliation  ${\mathcal L}_i$ satisfying \textsc{(a), (b)} and \textsc{(c)}. So, it is enough to prove that if $\epsilon >0$ is small enough, then any local  $C^1$-foliation ${\mathcal L}_i$ defined in a neighborhood of the rectangle $R_i$ and satisfying (\textsc{a})  and \textsc{(b)}, also satisfies \textsc{(d)} for some constant $K >0$.

In fact, choose and fix any point $x_i $ in the interior of the rectangle $R_i$. From the definition of rectangle, for each $z \in R_i$ there exists a unique point in the transversal intersection $$W^s_{R_i}(z) \pitchfork W_{R_i}^u(x_i) \neq \emptyset.$$ By continuity of the transversal intersection between $C^1$-manifolds, there exists $\epsilon' >0$ such that the following assertion holds:

 If $\mbox{dist}(x_i, y) < \epsilon'$ and if ${\mathcal L}_i$ is any local   foliation whose leaves have dimension $\mbox{dim}(F)$, are $C^1$, and  are   $\epsilon'$-near the unstable local leaves of $R_i$ in the $C^1$-topology, then for each $z \in R_i$ the intersection
 $W^s_{R_i}(z) \pitchfork {\mathcal L}_i(y)$ is   transversal and contains a single point.

 In particular,  we obtain:  \begin{equation} \label{eqnL24} W^s_{R_i}(z) \pitchfork {\mathcal L}_i(y) \neq \emptyset  \end{equation} $$ \forall \ y \in W^s_{R_i}(x_i) \mbox{ such that }  \mbox{dist}(x_i, y) < \epsilon', \ \ \forall \   z \in  R_n(x)  \subset R_i, \ \ \forall \   x \in R_i.$$

 Define $$A^s_i := \Big \{y \in W^s_{R_i}(x_i) \colon \mbox{dist}(x_i, y) < \epsilon'\Big\}.$$
 By construction $A^s_i$ is an open subset of the local stable submanifold $W^s_{R_i}(x_i)$, in the topology of this submanifold.  Construct a   real number   $K_i > 0$ large enough so
 \begin{equation} \label{eqnL26}\Leb^{ W^s(x_i)} \big(A_i^s  \big ) \geq K_i^{-1}. \end{equation}

 From assertion (\ref{eqnL24}) we deduce
 $$f^n \Big({\mathcal L}_i(y) \cap R_n(x)\Big) \pitchfork W^s_{f^n(R_n(x))}(w) \neq \emptyset$$
 for all $y \in A^s_i$ and for all $w \in  f^n(R_n(x)).$

 From the definition of the dynamical rectangle $R_n(x)$ and from the properties of the Markov partition, there exists a  rectangle $R_j \ni f^n(x)$ of the partition such that   local stable manifold $W^s_{R_j}(w) \supset W^s_{f^n(R_n(x))}(w)$ for all $w \in R_j$.

 So, we deduce   $$f^n \Big({\mathcal L}_i(y) \cap R_n(x)\Big) \pitchfork W^s_{R_j}(w) \neq \emptyset$$
 for all $y \in A^s_i$ and for all $w \in  R_j.$
 In other words, the pseudo-unstable $\mbox{dim}(F)-$ submanifold $f^n \Big({\mathcal L}_i(y) \cap R_n(x)\Big)$ intersects transversally all the local stable submanifolds of the rectangle $R_j$ where it is contained.  Thus, applying inequality (\ref{eqnL25}) we have
 \begin{equation} \label{eqnL27} \mbox{Leb}^{f^n({\mathcal L}_i)}f^n \Big(\big({\mathcal L}_i(y) \cap R_n(x)\big)\Big) \geq \frac{1}{K_{R_j}}.\end{equation}

  Finally, define $K = \max_{1 \leq i \leq k}\{K_{R_i}, \ K_i\}$. From inequalities (\ref{eqnL26}) and (\ref{eqnL27}), we conclude assertion \textsc{d)}, as wanted.
\end{proof}

\subsection{End of the proofs of Theorems \ref{mainTheo} and \ref{mainTheorem3}}

For any probability measure $\mu$ recall     equality (\ref{eqnL16}), defining the measurable set $A_{\epsilon, \, n}(\mu)$ which we called the $\epsilon$-pseudo basin of $\mu$ up to time $n$.
We will end the proof   of Theorem \ref{mainTheo}, by applying the following key result which bounds from below the Lebesgue measure of the set $A_{\epsilon, \, n}(\mu)$ for any ergodic measure $\mu$:

 \begin{Thm}
 \label{Theorem2}
 Let $M$ be a compact Riemannian manifold of finite dimension.
Let $f \in \mbox{Diff}^1(M)$ be Anosov with  hyperbolic splitting $TM= E \oplus F$, where $E$ and $F$ are the stable and unstable sub-bundles respectively. Let $\mu$ be an ergodic measure. Then:
\begin{equation}
\label{eqnL20-b} \lim_{\epsilon \rightarrow 0^+} \limsup_{n \rightarrow + \infty} \frac{\log \Leb (A_{\epsilon, \, n}(\mu))}{n} \geq h_{\mu}(f) - \int \log|\det Df |_{F } \, d \mu.
\end{equation}
\end{Thm}

  \begin{proof} We notice that the limit at left in equality (\ref{eqnL20-b}) does not depend on the choice of metric $\mbox{dist}^*$ that induces the weak$^*$ topology in the space ${\mathcal M}$ of Borel probability measures. In fact, to prove the latter assertion it is enough to argue as in the proof of part  (\textsc{d}) of Theorem \ref{propositionWeaklyPseudoPhysicalMeasures} in Section \ref{sectionPropertiesWeaklyPseudoPhysicalMeasures}.
   So, to prove Theorem \ref{Theorem2} we   choose and fix the following metric in ${\mathcal M}$:

   $$\mbox{dist}^*(\mu, \nu) := \sum_{i= 0}^{+ \infty} \frac{|\int \varphi_i \, d \mu - \int \varphi_i \, d \nu|}{2^i} \ \ \ \forall \ \mu, \nu \in {\mathcal M},$$
   where $\{\varphi\}_{i \in \mathbb{N}}$ is any fixed countable  family of real continuous functions $\varphi_i \in C^0(M, [0,1])$ that is dense in $C^0(M,[0,1])$. Note that, according to the metric $\mbox{dist}^*$, the balls are convex. In other words, if a finite number of probability measures belong to the ball with centre  $\mu$ and radius $\epsilon >0$, then any convex combination of those measures also belongs to it.

   For any $x \in M$ and for any natural number $n \geq 1$ denote: $$\psi(x):= \log \big|\det (Df_x|_{F(x)})\big|,$$ $$\psi_n(x):= \log\big|\det (Df^{n}_x|_{F(x)})\big| = \sum_{j= 0}^{n-1} \psi (f^j(x)) = n \cdot \int \psi \, d\sigma_n(x),$$
  where $\sigma_n(x)$ is the empiric probability constructed in Definition \ref{definitionEmpiricProba}.

    Fix any real value $\epsilon >0$. The real funcion $\psi: M \mapsto \mathbb{R}$ is   continuous because $f$ is of class $C^1$ and the sub-bundle $F$ is continuous. Thus, from the definition of the weak$^*$ topology in the space ${\mathcal M}$ of probability measures, we deduce that there exists $0 <\epsilon'< \epsilon  $ such that
    $$\nu, \mu \in {\mathcal M}, \ \ \mbox{dist}^*(\nu, \mu) < \epsilon' \ \ \Rightarrow \ \ \Big|\int \psi \, d \nu - \int \psi \, d \mu\Big| < \epsilon.$$
    In particular, for $\nu = \sigma_n(x)$ we deduce:
    \begin{equation} \label{eqnL30}\mbox{If } \mbox{dist}^*(\sigma_n(x), \mu) < \epsilon', \mbox{ then } \ \ \Big|\log\big|\det (Df^{n}_x|_{F(x)})\big| - n \cdot  \int \psi \, d \mu\Big| < n \cdot \epsilon.\end{equation}

    Since $\mu$ is an ergodic probability measure, we have $\lim_{n \rightarrow + \infty} \sigma_n(x) = \mu$ for $\mu$-a.e. $x \in M$. So, for the fixed value of $\epsilon' >0$ as above, and for $\mu$-a.e. $x \in M$, there exists $N(x) \geq 1$ such that $$\mbox{dist}^*(\sigma_n(x), \mu) < \epsilon'/2 \ \ \ \forall \, n \geq N(x).$$
    For any natural value of $N \geq 1$, define the set \begin{equation}
    \label{eqnL31}
     A_N := \Big \{x \in M: \ \ \mbox{dist}^*(\sigma_n(x), \mu) < \epsilon'/2 \ \ \ \forall \, n \geq N \Big\}.\end{equation}
    Since $A_N \subset A_{N+1}$ and $\mu\Big(\bigcup A_N \Big)= 1$, there exists $N\geq 1$ such that
    \begin{equation}
    \label{eqnL32}\mu(A_N) \geq 1-\epsilon.\end{equation}
    In the sequel, we fix such a value of $N \geq 1$.


     From the definition of the metrizable weak$^*$-topology in the space ${\mathcal M}$ of Borel-probability measures, it is standard to check that the Dirac delta probability $ \delta_x  $ depends uniformly continuously on the point $x \in M$. Since the empiric probability $\sigma_n(x)$ is a convex combination of Dirac delta measures, and the balls  in ${\mathcal M}$ are convex, we deduce that there exists $\delta_2 >0$ such that, for any pair of points $x,y \in M$ and for any natural value of $n \geq 1$, the following assertion holds:

     \begin{equation}
     \label{eqnL35}
     \mbox{If } \ \mbox{dist}(f^j(x), f^j(y)) < \delta_2 \ \mbox{ for all }   0 \leq j \leq n-1, \ \mbox{ then } \ \mbox{dist}^*(\sigma_n(x), \sigma_n(y)) < \epsilon'/2.
     \end{equation}

     For the fixed value of $\epsilon >0 $ at the beginning, we   construct the real numbers $0<\delta_0 < \alpha$  (where $\alpha$ is expansivity constant), and $K >0$, as in   Lemma \ref{lemmaFoliation}. We consider any  Markov partition ${\mathcal R} = \{R_i\}_{1 \leq i \leq k}$ with diameter smaller than $\min\{\delta_0, \delta_1, \delta_2\}$ and, for each rectangle $R_i$, we construct the $C^1$-foliation ${\mathcal L}_i$ that satisfies   the properties (\textsc{a})   to (\textsc{d}) of Lemma \ref{lemmaFoliation}.

       From equality (\ref{eqnL31}), assertion (\ref{eqnL35}), and the triangle property of the metric, we deduce the following assertion for all $n \geq N$:
     $$\mbox{If } x \in A_N \mbox{ and } y \in R_n(x), \ \ \mbox{ then } $$ \begin{equation} \label{eqnL37} \mbox{dist}^*(\sigma_n(y), \mu) \leq \mbox{dist}^*(\sigma_n(y), \sigma_n(x)) + \mbox{dist}^*(\sigma_n(x), \mu) <  \frac{\epsilon'}{2} + \frac{\epsilon'}{2} = \epsilon'< \epsilon .\end{equation}
     Recalling equality (\ref{eqnL16}), from the above assertion we deduce that $y \in A_{  \epsilon, \, n}(\mu)$ for all $y \in R_n(x)$.
     Since the rectangle $R_n(x)$ is any piece of the partition ${\mathcal R}_n = \bigvee_{j= 0}^{n-1} {\mathcal R}$ that intersects $A_N$, we deduce the following statement for all $n \geq N$:
     $$\mbox{If } Y \in {\mathcal R}_n \mbox{ and } Y \cap A_N \neq \emptyset, \mbox{ then } Y \subset A_{  \epsilon, \, n}(\mu).$$
     Therefore
     \begin{equation} \label{eqnL36}\Leb\Big(A_{\epsilon, \, n} (\mu)\Big) \geq \sum_{\displaystyle Y \in {\mathcal R}_n, \ Y \cap A_N \neq \emptyset} \Leb(Y).\end{equation}
     Besides, joining  assertion  (\ref{eqnL30}) and inequality (\ref{eqnL37}), we deduce the following property for all $n \geq N$:
     $$\mbox{If } Y \in {\mathcal R}_n \mbox{ and } Y \cap A_N \neq \emptyset, \mbox{ then } $$
     \begin{equation}
     \label{eqnL38} \Big|\log\big|\det (Df^{n}_y|_{F(y)})\big| - n \cdot \int \psi \, d \mu\Big| < n \cdot \epsilon \ \ \ \forall \ y \in Y.\end{equation}

     Now, for any $n \geq N$, let us compute $\Leb(Y)$ for any rectangle $Y \in {\mathcal R}_n$ such that $Y \cap A_N \neq \emptyset$. Since $Y \subset R_i \in {\mathcal R}$, to compute $\Leb(Y)$ we will use the Fubini decomposition of  the Lebesgue measure along the local pseudo-unstable $C^1$-foliation ${\mathcal L}_i$. Applying part  (\textsc{d}) of Lemma \ref{lemmaFoliation} consider the point $x_i \in R_i$ and the  submanifold $A^s_i \subset W^s_{R_i}(x_i)$.  Taking the Fubini decomposition of $\Leb$ we obtain:

     $$\Leb(Y) = \int_{z \in W^s_{R_i}(x_i)} \, d \Leb^{W^s(x_i)}(z) \int_{y \in {\mathcal L}_i(z) \cap Y} |\det D\phi_i(y)| \, d \Leb^{{\mathcal L}_i(z)}(y) ,$$
     where $\phi_i^{-1}$ is a local $C^1$-diffeomorhism that parameterizes   the neighborhood of $R_i$ and trivializes the $C^1$-foliation ${\mathcal L}_i$. Therefore, $|\det D\phi_i|$ is continuous and bounded away from zero by a constant, say $k_i >0$.
     Since $A^s_i $ is an open subset of $W^s_{R_i}(x_i)$ in the topology of this local stable manifold, we obtain:
     $$ \Leb(Y) \geq k_i \cdot    \int_{z \in A^s_i} \, d \Leb^{W^s(x_i)}(z) \int_{y \in {\mathcal L}_i(z) \cap Y}   d \Leb^{{\mathcal L}_i(z)}(y).$$
     Changing variables $y' = f^n(y) $ in the  integral at right, we obtain:
     $$ \Leb(Y)  \geq  k_i \cdot    \int_{z \in A^s_i} I(z) \, d \Leb^{W^s(x_i)}(z), \mbox{ where } $$ $$I(z) = \int_{y' \in f^n\big({\mathcal L}_i(z) \cap Y \big)}  \Big|\det Df^{n}_{f^{-n}(y')}|_{\displaystyle T_{f^{-n}(y')} {\mathcal L}_i (z)} \, \Big|^{-1} d \Leb^{f^n({\mathcal L}_i(z)}(y').$$
      Since $y \in Y \subset R_i$, we can apply inequality at left of part  \textsc{c}) of Lemma \ref{lemmaFoliation}:
     $$ \Leb(Y)  \geq   k_i  \cdot K^{-1} \cdot  e^{-n \epsilon} \cdot \int_{z \in A^s_i}  J(z) \,   d \Leb^{W^s(x_i)}(z), \mbox{ where } $$ $$J(z)=  \int_{y' \in f^n\big({\mathcal L}_i(z) \cap Y \big)}  \Big|\det Df^{n}_{f^{-n}(y')}|_{\displaystyle F({f^{-n}(y'))} } \, \Big|^{-1} d \Leb^{f^n({\mathcal L}_i(z)}(y').$$

     Since $y =f^{-n}(y') \in Y$ and $Y \cap A_N \neq \emptyset$, we can apply inequality (\ref{eqnL38}):
      $$ \Leb(Y) \geq  k_i  \cdot K^{-1} \cdot  e^{\displaystyle -2n \epsilon -n \int \psi \, d \mu} \cdot J, \  \mbox{ where }$$ $$J= \int_{z \in A^s_i} \, d \Leb^{W^s(x_i)}(z) \int_{y' \in f^n\big({\mathcal L}_i(z) \cap Y \big)}   d \Leb^{f^n({\mathcal L}_i(z)}(y') = $$
     $$   \int_{z \in A^s_i} \Leb^{f^n({\mathcal L}_i(z)}\big( f^n({\mathcal L}_i(z)\big )\, d \Leb^{W^s(x_i)}(z).$$
     From part (\textsc{d}) of Lemma \ref{lemmaFoliation} we know that $\Leb^{f^n({\mathcal L}_i(z)}\big( f^n({\mathcal L}_i(z)\big ) \geq K^{-1}$ for all $z \in A^s_i$, and besides   $\Leb^{W^s(x_i)}(A^s_i) \geq K^{-1}$.  Thus, we have proved the following inequality for all $n \geq N$, and for all $Y \in {\mathcal R}_n$ such that $Y \cap A_N \neq \emptyset$:
     $$ \Leb(Y) \geq  k_i  \cdot K^{-3} \cdot  e^{\displaystyle -2n \epsilon -n \int \psi \, d \mu}.$$
     Joining the above inequality with inequality (\ref{eqnL36}), we deduce, for all $n \geq N$:
     $$\Leb(A_{\epsilon, \, n}(\mu)) \geq k_i  \cdot K^{-3} \cdot  e^{\displaystyle -2n \epsilon -n\int \psi \, d \mu + \log \# \{Y \in {\mathcal R}_n \colon Y \cap A_N \neq \emptyset\}}.$$
      Therefore,
     $$ \limsup_{n \rightarrow + \infty} \frac{\log \Leb(A_{\epsilon, \, n}(\mu))}{n}  \geq $$ $$-2 \epsilon - \int |\det Df|_F| \, d \mu  + \limsup_{n \rightarrow + \infty} \frac{\log \# \{Y \in {\mathcal R}_n \colon Y \cap A_N \neq \emptyset\}}{n}. $$
     Finally, applying Lemma \ref{lemmaEntropia}, we deduce that
     $$\limsup_{n \rightarrow + \infty} \frac{\log \Leb(A_{\epsilon, \, n}(\mu))}{n}  \geq $$ $$-2 \epsilon   - \int |\det Df|_F| \, d \mu + \limsup_{n \rightarrow + \infty} \frac{H({{\mathcal R}_n, \mu})}{n} -  K_0 \cdot \epsilon    \ \ \ \forall \ 0 < \epsilon < 1. $$
     So, from equality (\ref{eqnMetricEntropy}), we conclude
     $$\lim_{\epsilon \rightarrow 0^+} \limsup_{n \rightarrow + \infty} \frac{\log \Leb(A_{\epsilon, \, n}(\mu))}{n}  \geq h_{\mu}(f) -   \int |\det Df|_F| \, d \mu ,$$
     ending the proof of Theorem \ref{Theorem2}.
  \end{proof}

\vspace{.5cm}

Now, we are ready to end the proofs of Theorems \ref{mainTheo} and \ref{mainTheorem3}, as   consequences of Theorems \ref{Theorem2} , \ref{TheoremConvexHull} and \ref{TheoremCCE1}:

 \vspace{.3cm}

 \noindent \textsc { Part b) of Theorem \ref{mainTheo}, necessary condition: } \em If $f \in \mbox{\em Diff}^1(M)$ is Anosov and if $\mu$ is an invariant measure satisfying Pesin's Entropy Formula, then its ergodic components $\mu_x$ are weak pseudo-physical for $\mu$-a.e. $x \in M$.   \em

\begin{proof}
First, let us assume that $\mu$ is ergodic satisfying Pesin's Entropy Formula. From Theorem \ref{Theorem2} we obtain $$\lim_{\epsilon \rightarrow 0^+}  \limsup_{n \rightarrow + \infty} \frac{\log \Leb\Big(A_{\epsilon, \, n}(\mu) \Big)}{n} \geq 0.$$
From equality (\ref{eqnL16}), if $\epsilon_1 < \epsilon_2$ then $A_{\epsilon_1, \, n}(\mu) \subset A_{\epsilon_2, \, n}(\mu)$.

So   $\displaystyle \limsup_{n \rightarrow + \infty} \frac{\log \Leb\Big(A_{\epsilon, \, n}(\mu) \Big)}{n}$  is increasing with $\epsilon >0$. Thus
$$ \limsup_{n \rightarrow + \infty} \frac{\log \Leb\Big(A_{\epsilon, \, n}(\mu) \Big)}{n} \geq 0 \ \ \forall \ \epsilon >0.$$
But since $\Leb$ is a probability measure, we conclude that
$$\limsup_{n \rightarrow + \infty} \frac{\log \Leb\Big(A_{\epsilon, \, n}(\mu) \Big)}{n} = 0 \ \ \forall \ \epsilon >0. $$
Applying Definition \ref{definitionWeakPseudoPhysicalMeasure}, we deduce that $\mu$ is weak pseudo-physical.

We have proved that any ergodic measure that satisfies Pesin's Entropy Formula is weak pseudo-physical. Now let us consider a non ergodic measure $\mu$  that satisfies Pesin's Entropy Formula. From Theorem \ref{TheoremConvexHull} we know that its ergodic components $\mu_x$ also satisfy that formula for $\mu$-a.e. $x \in M$. We conclude that   the ergodic components $\mu_x$ of $\mu$ are weak pseudo-physical for $\mu$-a.e. $x \in M$, as wanted.
\end{proof}

\vspace{.3cm}

 Finally to complete all the proofs, we add the following immediate end:

 \vspace{.3cm}

\noindent \textsc { End of the proof of  Theorem \ref{mainTheorem3}}.

\begin{proof}
The equality of Theorem \ref{mainTheorem3}  is immediately obtained by joining the  inequalities   of Theorems \ref{TheoremCCE1} and \ref{Theorem2}.
\end{proof}


\bibliographystyle{amsalpha}

\end{document}